\documentclass[10pt]{article}

\usepackage[english]{babel}
\usepackage[utf8]{inputenc}
\usepackage[T1]{fontenc}

\usepackage{amsmath, amsthm, amssymb, mathtools}

\usepackage[shortlabels]{enumitem}

\usepackage{cases}

\AtBeginDocument{\def\MR#1{}}

\usepackage{hyperref}

\usepackage{sectsty}

\usepackage{tikz}
\usetikzlibrary{positioning}
\usepackage{caption}

\usepackage[a4paper]{geometry}

\providecommand{\keywords}[1]{\textbf{Keywords.} #1}
\providecommand{\MSC}[1]{\textbf{2010 Mathematics Subject Classification.} #1}

\newtheorem{theorem}{Theorem}[section]

\newtheorem{lemma}[theorem]{Lemma}
\newtheorem{proposition}[theorem]{Proposition}

\theoremstyle{definition}
\newtheorem{definition}[theorem]{Definition}
\newtheorem{remark}[theorem]{Remark}

\renewcommand\epsilon{\varepsilon}

\usepackage{color}

\usepackage{xspace}

\newcommand{\R}{\field{R}\xspace}

\newcommand{\N}{\field{N}\xspace}

\newcommand{\field}[1]{\ensuremath{\mathbb{#1}}}

\newcommand{\ens}[1]{ \left\{#1\right\} }

\newcommand\diag{\mathrm{diag} \,}

\newcommand{\Tinf}{T_{\mathrm{inf}}}
\newcommand\TCN{T_{\scriptscriptstyle \mathrm{[CN]} }}

\newcommand\pt[1]{\frac{\partial #1}{\partial t}}
\newcommand\px[1]{\frac{\partial #1}{\partial x}}

\newcommand\pxi[1]{\frac{\partial #1}{\partial \xi}}

\newcommand\Tau{\mathcal{T}}

\newcommand\ssin{s^{\mathrm{in}}}

\newcommand{\rank}{\mathrm{rank} \,}

\newcommand\Id{\mathrm{Id}}

\newcommand\abs[1]{\left|#1\right|}
\newcommand\st{\quad \middle| \quad}

\newcommand\ddx{\frac{d}{dx}}
\newcommand\setCN{\mathcal{B}}

\newcommand\dds{\frac{d}{ds}}

\usepackage{ifthen}
\newcommand{\syst}[2]{
\ifthenelse{\equal{#2}{}}{\left(\Lambda,#1,Q\right)}
{\ifthenelse{\equal{#2}{b}}{\left(\Lambda,-,Q,#1\right)}{}}
{\ifthenelse{\equal{#2}{c}}{\left(\Lambda,-,Q^0,#1\right)}{}}
}

\newcommand\clos[1]{\overline{#1}}

\newcommand\tr{{\mathsf{T}}}

\newcommand\nmin{n^*}

\newcommand\opD{\mathcal{D}}

\title{A method to determine the minimal null control time of 1D linear hyperbolic balance laws}

\author{
Long Hu\thanks{School of Mathematics, Shandong University, Jinan, Shandong 250100, China.  E-mail: \texttt{hul@sdu.edu.cn}}
\and
Guillaume Olive\thanks{Faculty of Mathematics and Computer Science, Jagiellonian University, ul. {\L}ojasiewicza 6, 30-348 Krak\'{o}w, Poland. E-mail: \texttt{math.golive@gmail.com} or \texttt{guillaume.olive@uj.edu.pl}}
}

\date{\today}

\begin{document}

\maketitle

\begin{abstract}
In this paper we introduce a method to find the minimal control time for the null controllability of 1D first-order linear hyperbolic systems by one-sided boundary controls when the coefficients are regular enough.
\end{abstract}

\keywords{Hyperbolic systems; Minimal control time; Null controllability}

\vspace{0.2cm}
\MSC{35L40; 93B05}

\tableofcontents

\section{Introduction}

\subsection{Problem description}

In this paper, we are interested in the null controllability of a class of one-dimensional (1D) first-order linear hyperbolic systems (see e.g. \cite{BC16} for applications).
The equations of the system are
\begin{subequations}\label{syst}
\begin{equation}\label{syst:equ}
\pt{y}(t,x)+\Lambda(x) \px{y}(t,x)=M(x) y(t,x).
\end{equation}
Above, $t \in (0,T)$ is the time variable, $T>0$, $x \in (0,1)$ is the space variable and the state is $y:(0,T) \times (0,1) \to \R^n$ $(n \geq 2$).
The $n \times n$ matrix $\Lambda$ will always be assumed diagonal $\Lambda =\diag(\lambda_1,\ldots,\lambda_n)$, with $m \geq 1$ negative speeds and $p \geq 1$ positive speeds ($m+p=n$) such that:
\begin{equation}\label{hyp speeds}
\lambda_1(x)<\cdots<\lambda_m(x) <0<\lambda_{m+1}(x)<\cdots<\lambda_{m+p}(x),
\end{equation}
for every $x \in [0,1]$.
The $n \times n$ matrix $M$ couples the equations of the system inside the domain and will be called the internal coupling matrix.
We will assume the following regularity:
$$\Lambda, M \in C^{r+1}([0,1])^{n \times n} \quad \text{ for some } r \in \N.$$

The system will be evolving forward in time, so we consider an initial condition at time $t=0$:
\begin{equation}\label{syst:IC}
y(0,x)=y^0(x).
\end{equation}

Let us now discuss the boundary conditions.
The structure of $\Lambda$ induces a natural splitting of the state into components corresponding to negative and positive speeds, denoted respectively by $y_-$ and $y_+$.
For the above system to be well-posed in $(0,T) \times (0,1)$, we then need to add boundary conditions at $x=1$ for $y_-$ and at $x=0$ for $y_+$.
We will consider the following type of boundary conditions:
\begin{equation}\label{syst:BC}
y_-(t,1)=u(t), \quad y_+(t,0)=Qy_-(t,0).
\end{equation}
The function $u:(0,T) \to \R^m$ is called the control, it will be at our disposal.
It only acts on one part of the boundary and, on the other part of the boundary, the equations are coupled by the matrix $Q \in \R^{p \times m}$.
This matrix will be called the boundary coupling matrix.
\end{subequations}

In what follows, \eqref{syst:equ}, \eqref{syst:IC} and \eqref{syst:BC} will be referred to as system \eqref{syst}.
It is well-posed in several functional settings.
In this paper, we will work in $L^2$ for the state and the controls.
For every $T>0$, $y^0 \in L^2(0,1)^n$ and $u \in L^2(0,T)^m$, there exists a unique solution
$$
y \in C^0([0,T];L^2(0,1)^n) \cap C^0([0,1];L^2(0,T)^n).
$$
By solution we mean ``solution along the characteristics''.
We refer for instance to \cite{CHOS21} for a proof in such a setting (see also \cite[Appendix A]{BC16} when $u=0$).

The regularity $C^0([0,T];L^2(0,1)^n)$ of the solution allows us to consider control problems in the state space $L^2(0,1)^n$:

\begin{definition}
Let $T>0$ be fixed.
We say that system \eqref{syst} is null controllable in time $T$ if, for every $y^0 \in L^2(0,1)^n$, there exists $u \in L^2(0,T)^m$ such that the corresponding solution $y$ to system \eqref{syst} satisfies $y(T,\cdot)=0$.
\end{definition}

Since controllability in time $T_1$ implies controllability in any time $T_2 \geq T_1$, it is natural to try to find the smallest possible control time, the so-called ``minimal control time''.

\begin{definition}
We denote by $\Tinf \in [0,+\infty]$ the minimal null control time of system \eqref{syst}, that is
$$
\Tinf=
\inf\ens{T>0 \st \text{System \eqref{syst} is null controllable in time $T$}}.
$$
\end{definition}

The time $\Tinf$ is named ``minimal'' null control time according to the current literature, despite it is not always a minimal element of the set.
We keep this naming here, but we use the notation with the ``inf'' to avoid eventual confusions.
The goal of this article is to introduce a method to find $\Tinf$ for a large class of parameters $M,Q$.

In order to state the results of the literature, we need to introduce the following times:
$$
T_i=\int_0^1 \frac{1}{\abs{\lambda_i(\xi)}} \, d\xi, \quad 1 \leq i \leq n.
$$
The time $T_i$ is the minimal control time for single equation (the transport equation) with speed $\lambda_i$.
Note that the assumption \eqref{hyp speeds} implies in particular the following order relation:
$$
T_1 \leq \cdots \leq T_m \quad \text{ and } \quad
T_{m+1} \geq \cdots \geq T_{m+p}.
$$

Finally, we recall the key notion introduced in \cite[Section 1.2]{HO22-JDE} of ``canonical form'' for boundary coupling matrices.

\begin{definition}\label{def can form}
We say that a matrix $Q^0 \in \R^{p \times m}$ is in canonical form if it has at most one nonzero entry on each row and each column, and this entry is equal to $1$.
We denote by $(r_1,c_1), \ldots, (r_\rho,c_\rho)$ the positions of the corresponding nonzero entries, with $r_1<\ldots<r_\rho$.
\end{definition}

We can prove that, for every $Q \in \R^{p \times m}$, there exists a unique $Q^0 \in \R^{p \times m}$ in canonical form such that $LQU=Q^0$ for some lower triangular matrix $L \in \R^{p \times p}$ with diagonal entries all equal to one and some invertible upper triangular matrix $U \in \R^{m \times m}$.
The matrix $Q^0$ is called the canonical form of $Q$ and we can extend the definition of $(r_1,c_1), \ldots, (r_\rho,c_\rho)$ to any nonzero matrix.
We refer to the above reference for more details.

\subsection{Literature}\label{sect literature}

The controllability of 1D first-order hyperbolic systems with boundary controls has been widely studied in the literature.
We only recall here some important results regarding the control time of such systems, since it is the primary focus of the present paper.
All the results quoted below are in fact valid for matrices $\Lambda$ and $M$ with $C^{0,1}$ and $L^\infty$ regularity only, respectively.

\begin{itemize}
\item
It was first proved in the celebrated survey \cite{Rus78} that system \eqref{syst} is null controllable in time $T_{m+1}+T_m$.
A strength of this result is that it is valid for any $M$ and $Q$.
However, it was also observed in that paper that the minimal control time can be smaller than $T_{m+1}+T_m$.
Finding the minimal control time even in the simpler case $M=0$ was then left as an open problem.

\item
For $M=0$, the minimal null control time was eventually found in \cite{Wec82}.
The author gave an explicit formula of this time in terms of some indices related to $Q$.
A simpler formula is however now known from the work \cite{HO22-JDE} (see Section 5.3 therein):
$$M=0 \quad \Longrightarrow \quad \Tinf=\max_{1 \leq k \leq \rho} \ens{T_{m+r_k}+T_{c_k}, \quad T_m, \quad T_{m+1}},$$
where the indices $(r_k,c_k)$ refer to $Q$.

\end{itemize}

After the seminal work \cite{Rus78}, it seems that the community turned its attention to quasilinear systems in the $C^1$ framework of so-called semi-global solutions (see e.g. \cite{LR03,Li10,Hu15}).
However, lately there has been a resurgence on finding the minimal control time for linear systems when $M \neq 0$.
This was initiated by the authors in \cite{CN19} and followed by a series of works \cite{HO21-JMPA,CN21,HO21-COCV,CN21-pre,HO22-JDE,HO23-pre}.

\begin{itemize}
\item
In \cite{CN19,CN21}, the authors introduced the following class of boundary coupling matrices:
\begin{equation}\label{def setCN}
\setCN=
\ens{Q \in \R^{p \times m} \st \text{$Q(i)$ is invertible for all $i \leq \min\ens{p,m-1}$}},
\end{equation}
where $Q(i)$ denotes the $i \times i$ matrix formed from the first $i$ rows and columns of $Q$ ($\setCN=\R^p$ for $m=1$).
The authors then proved that, if $Q \in \setCN$, then the system \eqref{syst} is null controllable in any time $T>\TCN$, where
$$
\TCN=
\max_{1 \leq k \leq \min\ens{m,p}} \ens{T_{m+k}+T_k, \quad T_m}.
$$
This result is valid for any $M$.
Another important feature of this result is that it is stable by small perturbations of $Q$.
Note however the time $\TCN$ can be strictly larger than the minimal null control time: consider for instance $M=0$ and $Q=\begin{pmatrix} 1 & 0 \\ 0 & 0 \end{pmatrix} \in \setCN$, with $\Lambda$ such that $T_4+T_2>T_3+T_1$, for which we have $\TCN=T_4+T_2>\Tinf=\max\ens{T_3+T_1,T_2}$.

\item
For full row rank boundary coupling matrices ($\rank Q=p$), we found the minimal null control time in \cite{HO21-JMPA}.
In this case, we showed that this time is the same as for the system without internal coupling ($M=0$).
More precisely, we proved the formula
$$
\Tinf=
\max_{1 \leq k \leq p} \ens{T_{m+k}+T_{c_k}, \quad T_m}
\quad \text{ if } \rank Q=p,
$$
where the indices $c_k$ refer to $Q$.

%
%
%

\item
For systems of $n=2$ equations, the minimal null control time was found in \cite{CVKB13} for $Q \neq 0$ and in \cite{HO21-COCV} for $Q=0$.
Notably, we showed in the second situation that this time can potentially be any number between $\max\ens{T_1,T_2}$ and $T_2+T_1$, depending on the values of $M$.

\item
In \cite{HO22-JDE}, we found the smallest and largest values that the minimal null control time can take with respect to the internal coupling matrix $M$.
More precisely, we showed therein that $\Tinf$ is never smaller than the minimal null control time of the system without internal coupling, and that
\begin{equation}\label{max Tinf}
\max_M \, \Tinf=
\max_{1 \leq k \leq \rho_0} \ens{
T_{m+k}+T_{c_k}, \quad T_{m+\rho_0+1}+T_m
},
\end{equation}
where $\rho_0$ is the largest integer $i \leq p$ such that the $i \times m$ matrix formed from the first $i$ rows of $Q$ has rank $i$ ($\max_M \, \Tinf=T_{m+1}+T_m$ if the first row of $Q$ is zero).
As explained in \cite[Remark 1.17]{HO22-JDE}, these results generalize all the aforementioned results, apart from the case $n=2$.

\item
Finally, in \cite{HO23-pre}, we characterized the minimal null control time in the case $m=1$ for a large class of constant internal coupling matrices $M$ (e.g. constant matrices with zero diagonal entries).
\end{itemize}


Despite these numerous results, there are still situations where the minimal null control time cannot be determined (e.g. when the first row of $Q$ is zero).
The goal of this article is to introduce a method to find this time in some new situations.

\paragraph{Organization of the paper.}
The rest of this article is organized as follows.
Section \ref{sect prelim} is a preparatory section where we recall some results and prove an important lemma.
In Section \ref{sect method}, we present our method to compute the minimal null control time of systems related to system \eqref{syst}.
In Section \ref{sect derivatives GM}, we explain how to use this method to also find the minimal null control time of system \eqref{syst}.
We then conclude the article with a summary of the method and several applications.

\section{Preliminary results}\label{sect prelim}

The presentation of our method requires some preparatory steps.
For the rest of this article, it is convenient to extend $\Lambda$ to a $C^{r+1}$ function of $\R$, still denoted by the same, so that the characteristics are globally defined.
Clearly, this can be done respecting the order relation \eqref{hyp speeds}.
Since all the results of the paper only depend on the values of $\Lambda$ in $[0,1]$, they do not depend on the chosen extension.

\subsection{Equivalent system by the backstepping method}\label{sect syst G}

First of all, it will be convenient to use the notion of ``equivalent systems'' introduced in \cite{HO22-JDE} (see also \cite{HO21-JMPA}).
Following \cite[Definition 1.18]{HO22-JDE}, we will say that two systems are equivalent if we can map the solutions of one system onto those of the other system by means of an invertible bounded linear transformation of the state space.
Clearly, two equivalent systems have the same minimal null control time.

The first step is to use the so-called backstepping method for PDEs (see e.g. \cite{KS08}).
The goal of the backstepping method is to obtain an equivalent system with a simpler coupling structure.
This method was developed specifically for hyperbolic systems in a series of works \cite{CVKB13,DMVK13,HDM15,HDMVK16,HVDMK19,HO23-pre}.
It was also the starting point in the investigation of the null controllability properties of system \eqref{syst} in the recent works \cite{CN19,CN21,HO21-COCV,HO22-JDE,HO23-pre}.

For this method to be used, we first need to remove the diagonal entries of $M$ (this will be clear from what follows).
This can easily be done by using the transformation
$$\hat{y}(t,x)=D(x)y(t,x),$$
with $D(x)=\diag(e^{-\int_0^x \frac{m_{kk}(\xi)}{\lambda_k(\xi)} \,d\xi})_{1 \leq k \leq n}$.
Then, $\hat{y}$ satisfies the same boundary condition at $x=0$ as $y$ and the same equations as $y$ but with $M$ replaced by
\begin{equation}\label{def M0}
M^0(x)=\left(D(x)M(x) +\Lambda(x)D'(x)\right) D(x)^{-1}.
\end{equation}
Note that $M^0 \in C^{r+1}$ and that the diagonal entries of $M^0$ are all equal to zero.

Now the core of the backstepping method consists in using the Volterra integral transformation of the second kind
$$
\tilde{y}(t,x)=\hat{y}(t,x)-\int_0^x K(x,\xi)\hat{y}(t,\xi) \, d\xi,
$$
where $K:\clos{\Tau} \to \R^{n \times n}$ is a kernel to be determined, defined on the closure of the triangle $\Tau=\ens{(x,\xi) \in \R^2 \st 0<\xi<x<1}$.
It is not difficult to see, at least formally, that $\tilde{y}$ is solution to the system
\begin{equation}\label{syst G}
\begin{dcases}
\pt{\tilde{y}}(t,x)+\Lambda(x) \px{\tilde{y}}(t,x)=G(x) \tilde{y}_-(t,0), \\
\tilde{y}_-(t,1)=\tilde{u}(t), \quad \tilde{y}_+(t,0)=Q\tilde{y}_-(t,0), \\
\tilde{y}(0,x)=\tilde{y}^0(x),
\end{dcases}
\end{equation}
(for some control $\tilde{u}$ and initial data $\tilde{y}^0$), with $G:[0,1] \to \R^{n \times m}$ given by
\begin{equation}\label{def G}
G(x)
=
K(x,0)B, \quad B=-\Lambda(0)\begin{pmatrix} \Id_m \\ Q \end{pmatrix},
\end{equation}
provided that $K$ satisfies the so-called kernel equations:
\begin{equation}\label{kern equ}
\begin{dcases}
\Lambda(x)\px{K}(x,\xi)
+\pxi{K}(x,\xi)\Lambda(\xi)
+K(x,\xi) (\Lambda'(\xi)+M^0(\xi))=0,
\\
\Lambda(x) K(x,x)-K(x,x)\Lambda(x) =M^0(x).
\end{dcases}
\end{equation}
We refer for instance to \cite[Section 2.2]{HVDMK19} for some details.

The technical point in the backstepping method is to prove that there indeed exists a solution to the kernel equations \eqref{kern equ}.
Note that, because of the second condition in  \eqref{kern equ}, these equations have a solution only if the diagonal entries of $M^0$ are equal to zero.
This explains the need to change $M$ to $M^0$ by a preliminary transformation, as discussed before.
It then follows from the results of \cite{HDMVK16} that the kernel equations \eqref{kern equ} have many $C^{r+1}$ solutions in $\clos{\Tau}$ (see also the method in \cite[Appendix B]{HO23-pre}).
Recall that we assume that $\Lambda \in C^{r+1}$, so that the coefficients in \eqref{kern equ} are a priori only in $C^r$ because of the term $\Lambda'(\xi)$, but doing the change of unknown $\hat{K}(x,\xi)=K(x,\xi)\Lambda(\xi)$ we see that $\hat{K}$ satisfies a similar system without involving the derivative of $\Lambda$.
Note that we then have $G \in C^{r+1}([0,1])^{n \times m}$.

As a result, we see that the study of the controllability properties of our initial system \eqref{syst} comes down to the study of the controllability properties of systems of the form \eqref{syst G}.
Let us point out that the choice of solution to the kernel equations does not affect the controllability properties because all the corresponding systems \eqref{syst G}-\eqref{def G} are equivalent.
Now, two problems naturally arise:
\begin{enumerate}[(I)]
\item\label{prob 1}
Can we characterize the minimal null control time of systems of the form \eqref{syst G} in function of $Q$ and $G$ ?

\item\label{prob 2}
If so, can we deduce the minimal null control time of system \eqref{syst} in function of $Q$ and $M$, when $G$ takes the particular form \eqref{def G}-\eqref{kern equ} ?
\end{enumerate}

Partial answers can be found in the references mentioned in the literature section (Section \ref{sect literature}).
The goal of the present article is to introduce a general method that is able to tackle both problems for a large class of $G$ and $M$.

\subsection{Simplification of the coupling structure}\label{sect previous results}

Without further information on the structure of $G$, it is still difficult to characterize the minimal null control time of systems of the form \eqref{syst G}.
A lot of work has been done in \cite{HO22-JDE} to simplify the coupling structure of these systems in order to find this minimal time.
We collect some of the results of this paper in this section.

Below, we denote by $G_{--}$ (resp. $G_{+-}$) the submatrix formed by the first $m$ (resp. last $p$) rows of $G$.

The first result that we need is the following (\cite[Theorem 7.1]{HO22-JDE}):

\begin{theorem}
The minimal null control time of system \eqref{syst G} does not depend on $G_{--}$.
\end{theorem}

Consequently, we only have to focus on systems of the form \eqref{syst G} with $G_{--}=0$.
In the sequel such systems will be denoted by
$$(Q,G_{+-}).$$
We will denote the corresponding minimal null control time by $\Tinf(Q,G_{+-})$, or simply use $\Tinf$ when there is no ambiguity.
Finally, we will denote by $Q_j$ (resp. $G_j$) the $j$-th row of $Q$ (resp. $G_{+-}$).

The second result we need is the following precised statement of \cite[Proposition 4.1]{HO22-JDE} (which is contained in its proof):

\begin{lemma}\label{lem reduc Q}
For any invertible lower triangular matrix $L \in \R^{p \times p}$, the system $(Q,G_{+-})$ is equivalent to the system $(LQ,\tilde{G}_{+-})$, where $\tilde{G}_{+-}(x)$ is the matrix whose first row is $\tilde{G}_1(x)=l_{11}G_1(x)$ and whose $j$-th row for $j \geq 2$ is
$$\tilde{G}_j(x)=l_{jj} G_j(x)+ L_j' G^{\zeta}_{1:j-1}(x),$$
where $L_j'=\begin{pmatrix} l_{j1} & \cdots & l_{j,j-1} \end{pmatrix}$,
\begin{equation}\label{def Gzeta}
G^{\zeta}_{1:j-1}(x)=
\begin{pmatrix}
G_1(\zeta_{j1}(x)) \\
G_2(\zeta_{j2}(x)) \\
\vdots \\
G_{j-1}(\zeta_{j,j-1}(x))
\end{pmatrix},
\end{equation}
and $\zeta_{jk}$ is the solution to
$$
\begin{dcases}
\zeta_{jk}'(x)= \frac{\lambda_{m+k}(\zeta_{jk}(x))}{\lambda_{m+j}(x)}, \quad \forall x \in \R, \\
\zeta_{jk}(0)=0.
\end{dcases}
$$

\end{lemma}

The last result that we need is a consequence of \cite[Corollary 1.12]{HO22-JDE} (see also \cite[Remark 1.17]{HO22-JDE}):

\begin{theorem}\label{thm comp time}
Let $\nmin=\min\ens{p,m}$.
Assume that the first $\nmin$ rows of $Q$ are linearly independent.
Then, $\Tinf(Q,G_{+-})$ does not depend on $G_{+-}$, and it is given by
$$\Tinf(Q,G_{+-})=\max_{1 \leq k \leq \nmin} \ens{T_{m+k}+T_{c_k}, \quad T_m},$$
where the indices $c_k$ refer to $Q$.
\end{theorem}

\subsection{Derivative of zero rows}

The last ingredient that will be instrumental in our method is the following.

\begin{lemma}\label{lem derivative}
Assume that $Q_{j_0}=0$ and $G_{j_0} \in C^1([0,1])^{1 \times m}$ for some $j_0 \in \ens{1,\ldots,p}$.
Then, for any $T>0$, the system $(Q,G_{+-})$ is null controllable in time $T$ if, and only if, so is the system $(\hat{Q},\hat{G}_{+-})$, defined by the same $j$-th row as $(Q,G_{+-})$ for $j \neq j_0$ (i.e. $\hat{Q}_j=Q_j$ and $\hat{G}_j=G_j$ for $j \neq j_0$) and
$$
\hat{Q}_{j_0}=G_{j_0}(0), \quad \hat{G}_{j_0}(x)=(\opD_{j_0} G_{j_0})(x),
$$
where $\opD_{j_0}$ is the differential operator $\lambda_{m+j_0} \ddx$.

\end{lemma}

\begin{proof}
We are going to show that the null controllability conditions are the equivalent for the $m+j_0$ components of the states of the two systems.
To this end, let us denote by $\chi_{m+j_0}$ the corresponding characteristic, i.e. the solution $\chi$ to
$$
\begin{dcases}
\dds \chi(s;t,x)= \lambda_{m+j_0}(\chi(s;t,x)), \quad \forall s \in \R, \\
\chi(t;t,x)=x,
\end{dcases}
$$
and let $\ssin_{m+j_0}(t,x)$ be the solution $s$ to $\chi_{m+j_0}(s;t,x)=0$.
We assume that $T \geq T_{m+j_0}$, otherwise neither system is controllable (see \cite[Lemma 3.3]{HO21-COCV}).
Then, the null controllability condition $y_{m+j_0}(T,\cdot)=0$ for the system $(Q,G_{+-})$ is equivalent to
$$
Q_{j_0} v(\ssin_{m+j_0}(T,x))
+\int_{\ssin_{m+j_0}(T,x)}^T  G_{j_0}(\chi_{m+j_0}(s;T,x)) v(s) \, ds
=0,
$$
for a.e. $x \in (0,1)$, where we introduced $v(s)=y_-(s,0)$.
We write $x=\chi_{m+j_0}(T;t,0)$ and obtain, equivalently,
\begin{equation}\label{equiv NC}
Q_{j_0} v(t)
+\int_t^T  G_{j_0}(\chi_{m+j_0}(s;t,0)) v(s) \, ds
=0,
\end{equation}
for a.e. $t \in (\ssin_{m+j_0}(T,1), T)$.
Using the assumption $Q_{j_0}=0$, we obtain the equivalent identity
$$
\int_t^T  G_{j_0}(\chi_{m+j_0}(s;t,0)) v_i(s) \, ds
=0,
$$
for every $t \in (\ssin_{m+j_0}(T,1), T)$.
Taking the derivative with respect to $t$, this is also equivalent to
$$
G_{j_0}(0) v(t)
+\int_t^T  G_{j_0}' (\chi_{m+j_0}(s;t,0))
\lambda_{m+j_0}(\chi_{m+j_0}(s;t,0))
v(s) \, ds
=0,
$$
for a.e. $t \in (\ssin_{m+j_0}(T,1), T)$.
This is the same condition as \eqref{equiv NC} with $Q_{j_0}$ and $G_{j_0}$ replaced by the new parameters $\hat{Q}_{j_0}=G_{j_0}(0)$ and $\hat{G}_{j_0}=\opD_{j_0}G_{j_0}$, as defined in the statement of the lemma.
\end{proof}

\begin{remark}
The $j_0$-th row of the new system $(\hat{Q},\hat{G}_{+-})$ can be obtained by formally applying the operator $\opD_{j_0}$ to the $(m+j_0)$-th component of the solution to the system $(Q,G_{+-})$.
More precisely, if $y$ is a solution to $(Q,G_{+-})$ and we define (formally)
$$
\hat{y}_-=y_-, \quad
\hat{y}_{m+j}
=\begin{dcases}
y_{m+j} & \text{ if } j \neq j_0, \\
\opD_{j_0} y_{m+j_0} & \text{ otherwise, }
\end{dcases}
$$
then $\hat{y}$ will be a solution to $(\hat{Q},\hat{G}_{+-})$.
Indeed, the equation and boundary condition satisfied by $y_{m+j_0}$, namely,
$$
\begin{dcases}
\pt{y_{m+j_0}}(t,x)+ (\opD_{j_0} y_{m+j_0})(t,x)=G_{j_0}(x)y_-(t,0),
\\
y_{m+j_0}(t,0)=0 \quad \text{ (since $Q_{j_0}=0$), }
\end{dcases}
$$
give the equation
\begin{align*}
\pt{(\opD_{j_0}y_{m+j_0})}(t,x)+ \opD_{j_0} (\opD_{j_0} y_{m+j_0})(t,x) &=(\opD_{j_0} G_{j_0})(x)y_-(t,0)
\\
&=\hat{G}_{j_0}(x)\hat{y}_-(t,0),
\end{align*}
and the boundary condition
\begin{align*}
(\opD_{j_0} y_{m+j_0})(t,0) &=-\pt{y_{m+j_0}}(t,0)
+G_{j_0}(0)y_-(t,0)
\\
 &=\hat{Q}_{j_0} \hat{y}_-(t,0).
\end{align*}

\end{remark}

\section{Method for the minimal control time of systems $(Q,G_{+-})$}\label{sect method}

We are now ready to present our method to determine the minimal null control time of systems of the form $(Q,G_{+-})$ and thus solve Problem \ref{prob 1}.

\subsection{General ideas}\label{sect gen ideas}

We first present the main ideas of the method.
Under a fundamental assumption, we are going to construct new systems $(Q^{[k]}, G^{[k]}_{+-})$, $k=1,2,\ldots,\nmin$, which have the same minimal null control time as the initial system $(Q,G_{+-})$, and such that the first $k$ rows of $Q^{[k]}$ are linearly independent.
This construction is as follows.

For convenience, we denote by $Q^{[0]}=Q$ and $G^{[0]}_{+-}=G_{+-}$.
Assume that $(Q^{[k-1]}, G^{[k-1]}_{+-})$ has been constructed with the desired properties for some $k \in \ens{1,\ldots,\nmin}$, and let us show how to construct $(Q^{[k]}, G^{[k]}_{+-})$.
We discuss two cases.

\begin{itemize}
\item\label{case nonzero}
If $Q^{[k-1]}_k$ and the first $k-1$ rows of $Q^{[k-1]}$ are linearly independent (resp. $Q^{[0]}_1 \neq 0$ for $k=1$), then we do nothing and set $Q^{[k]}=Q^{[k-1]}$, $G^{[k]}_{+-}=G^{[k-1]}_{+-}$.

\item
Otherwise, we construct a new system $(\hat{Q}^{[k-1]}, \hat{G}^{[k-1]}_{+-})$ in two steps as follows.

\begin{itemize}
\item
\textbf{Step 1: we put the $k$-th row to zero.}
If $k=1$, there is nothing to do and we set $\tilde{Q}^{[k-1]}=Q^{[k-1]}$, $\tilde{G}^{[k-1]}_{+-}=G^{[k-1]}_{+-}$.
Otherwise, this means that there exists a row vector $a \in \R^{1 \times (k-1)}$ such that
$$Q^{[k-1]}_k + a Q^{[k-1]}_{1:k-1}=0,$$
where $Q^{[k-1]}_{1:k-1}$ denotes the submatrix formed by the first $k-1$ rows of $Q^{[k-1]}$.
We then form the lower triangular matrix $L$ whose $k$-th row is $\begin{pmatrix} a & 1 & 0 & \cdots & 0 \end{pmatrix}$ and whose $j$-th row is $e_j^\tr$ for $j \neq k$ (here and in what follows, $e_j^\tr$ denotes the transpose of the $j$-th vector of the canonical basis of $\R^d$, with appropriate $d$).
We apply Lemma \ref{lem reduc Q} with such a $L$ to obtain a new system $(\tilde{Q}^{[k-1]}, \tilde{G}^{[k-1]}_{+-})$, which has the same minimal null control time as $(Q^{[k-1]},G^{[k-1]}_{+-})$, and such that the $k$-th row of $\tilde{Q}^{[k-1]}$ is now identically zero.

\item
\textbf{Step 2: we take the derivative of the $k$-th row.}
Since $\tilde{Q}^{[k-1]}_k=0$, we can apply Lemma \ref{lem derivative} to produce a new system $(\hat{Q}^{[k-1]}, \hat{G}^{[k-1]}_{+-})$ with the same minimal null control time as $(\tilde{Q}^{[k-1]}, \tilde{G}^{[k-1]}_{+-})$ (and thus as $(Q^{[k-1]}, G^{[k-1]}_{+-})$).
\end{itemize}

Then, we repeat the previous discussion with the new system $(\hat{Q}^{[k-1]}, \hat{G}^{[k-1]}_{+-})$ instead of $(Q^{[k-1]}, G^{[k-1]}_{+-})$.
Now comes our main assumption: we will assume that, after a finite number of iterations, we eventually end up in the first case, that is the $k$-th row of the new boundary coupling matrix and the first $k-1$ rows of $Q^{[k-1]}$ are linearly independent.

\end{itemize}

Applying the previous algorithm for each row of the system, we obtain in the end a new system $(Q^{[\nmin]}, G^{[\nmin]}_{+-})$ which has the same minimal null control time as the initial system $(Q,G_{+-})$, and such that the first $\nmin$ rows of $Q^{[\nmin]}$ are linearly independent.
Then, we conclude with Theorem \ref{thm comp time} that
$$\Tinf=\max_{1 \leq k \leq \nmin} \ens{T_{m+k}+T_{c_k}, \quad T_m},$$
where the indices $c_k$ refer to $Q^{[\nmin]}$.

\begin{remark}
We can always stop the algorithm before, at some row $k_0<\nmin$, but then we obtain only an upper bound for $\Tinf$.
Nevertheless, this still shows that the system $(Q,G_{+-})$ is null controllable in any time
$$
T>\max_{1 \leq k \leq k_0} \ens{T_{m+k}+T_{c_k}, \quad T_{m+k_0+1}+T_m},
$$
where the indices $c_k$ refer to $Q^{[k_0]}$.
This follows by replacing the use of Theorem \ref{thm comp time} by the upper bound \eqref{max Tinf}, which is also valid for systems of the form $(Q,G_{+-})$ (see \cite[Section 8.2]{HO22-JDE} for details).
\end{remark}

\begin{remark}
It has also been proved in \cite[Proposition 4.1]{HO22-JDE} that, for any invertible upper triangular matrix $U \in \R^{m \times m}$, the system $(Q,G_{+-})$ is equivalent to the system $(QU,G_{+-}U)$.
Using this property and Lemma \ref{lem reduc Q} we could, at the end of the construction of each system $(Q^{[k]},G^{[k]}_{+-})$, put the first $k$ rows of $Q^{[k]}$ in canonical form (Definition \ref{def can form}), up to changing $G^{[k]}_{+-}$.
\end{remark}

\subsection{Main result for systems $(Q,G_{+-})$}\label{sect main thm G}

Let us now precisely state the assumptions and conclusions of the method for a given row.
Below, for a matrix (or matrix-valued function) $A$ we denote by $A_{1:k}$ the submatrix formed by its first $k$ rows.

The main result of this section is the following.

\begin{theorem}\label{main thm G}
Consider a system $(Q,G_{+-})$ with $G_{+-} \in C^{s+1}$ ($s \in \N$).

\begin{itemize}
\item
Assume that
$$\gamma^l=0, \quad \forall l \in \ens{0,\ldots,s},$$
where
$$
\gamma^0=Q_1, \quad \gamma^l=G_1^{(l-1)}(0) \quad \text{ for $l \geq 1$}.
$$
Then, $\Tinf(Q,G_{+-})=\Tinf(\bar{Q},\bar{G}_{+-})$, where the system $(\bar{Q},\bar{G}_{+-})$ is defined by the same $j$ row as $(Q,G_{+-})$ for $j \neq 1$, and whose first row is
$$
\bar{Q}_1=
\lambda_{m+1}(0)^s G_1^{(s)}(0),
\quad
\bar{G}_1(x)=
(\opD_1^{s+1} G_1)(x).
$$

\item
Let $k \geq 2$ be fixed.
Assume that there exist some row vectors $a^0,\ldots,a^s \in \R^{1 \times (k-1)}$ such that
$$
\omega^l
+a^l Q_{1:k-1}=0, \quad \forall l \in \ens{0,\ldots,s},
$$
where
\begin{equation}\label{def omega}
\omega^0=Q_k, \quad
\omega^l=\bar{\omega}^l(0)
 \quad \text{ for $l \geq 1$},
\end{equation}
and
\begin{equation}\label{def omegabar}
\bar{\omega}^l(x)=
(\opD_k^{l-1} G_k)(x)
+\sum_{i=0}^{l-1} a^i (\opD_k^{l-1-i} G^{\zeta}_{1:k-1})(x).
\end{equation}
Then, $\Tinf(Q,G_{+-})=\Tinf(\bar{Q},\bar{G}_{+-})$, where the system $(\bar{Q},\bar{G}_{+-})$ is defined by the same $j$ row as $(Q,G_{+-})$ for $j \neq k$, and whose $k$-th row is
$$
\bar{Q}_k=\omega^{s+1},
\quad
\bar{G}_k(x)=(\opD_k \bar{\omega}^{s+1})(x).
$$

\end{itemize}

\end{theorem}

\begin{remark}\label{rem all dep on G at x=0}
To apply or iterate this result we only need to know the values of $G_{+-}$ and its derivatives at $x=0$.
\end{remark}


\begin{proof}[Proof of Theorem \ref{main thm G}]
We prove the result by induction on $s$.
\begin{itemize}
\item
Let us prove the first point.
For $s=0$ the result is obtained from Lemma \ref{lem derivative} since $\gamma^0=Q_1=0$.
Assume now that the result holds for $s \geq 0$ and let us prove it for $s+1$.
Using the induction assumption, we have $\Tinf(Q,G_{+-})=\Tinf(\bar{Q},\bar{G}_{+-})$, where $(\bar{Q},\bar{G}_{+-})$ is the system defined in the statement of the theorem.
Since in addition $\gamma^{s+1}=0$, we have $\bar{Q}_1=0$.
We then apply Lemma \ref{lem derivative} to obtain a new system, which has the same minimal null control time as $(\bar{Q},\bar{G}_{+-})$ and whose first row is
$$
\bar{G}_1(0)=(\opD_1^{s+1}G_1)(0), \quad \opD_1 \bar{G}_1=\opD_1^{s+2} G_1,
$$
(the other rows are unchanged).
Finally, we note that
$$(\opD_1^{s+1} G_1)(0) =\lambda_{m+1}(0)^{s+1} G_1^{(s+1)}(0),$$
since $G_1^{(l)}(0)=\gamma^{l+1}=0$ for every $l \leq s$.
This proves the induction.

\item
Let us prove the second point.
Let then $k \geq 2$ be fixed.
We first prove the property for $s=0$.
We apply Lemma \ref{lem reduc Q}, with $L$ defined by $L_k=\begin{pmatrix} a^0 & 1 & 0 & \cdots & 0 \end{pmatrix}$ and $L_j=e_j^\tr$ for $j \neq k$, to obtain a new system, which has the same minimal null control time as $(Q,G_{+-})$, whose first $k$ rows are
$$
\begin{pmatrix}
Q_{1:k-1} \\
Q_k +a^0 Q_{1:k-1}
\end{pmatrix}=
\begin{pmatrix}
Q_{1:k-1} \\
0
\end{pmatrix}
, \quad
\begin{pmatrix}
G_{1:k-1} \\
G_k +a^0 G^{\zeta}_{1:k-1}
\end{pmatrix}
=
\begin{pmatrix}
G_{1:k-1} \\
\bar{\omega}^1
\end{pmatrix},
$$
and with other rows unchanged.
To conclude we simply apply Lemma \ref{lem derivative}.
Assume now that the result holds for $s \geq 0$ and let us prove it for $s+1$.
Using the induction assumption, we have $\Tinf(Q,G_{+-})=\Tinf(\bar{Q},\bar{G}_{+-})$, where $(\bar{Q},\bar{G}_{+-})$ is the system defined in the statement of the theorem.
Since in addition $\omega^{s+1}
+a^{s+1} Q_{1:k-1}=0$, we have $\bar{Q}_k+a^{s+1} Q_{1:k-1}=0$.
We then apply Lemma \ref{lem reduc Q} to obtain a new system, which has the same minimal null control time as $(\bar{Q},\bar{G}_{+-})$, whose first $k$ rows are
$$
\begin{pmatrix}
Q_{1:k-1} \\
0
\end{pmatrix}
, \quad
\begin{pmatrix}
G_{1:k-1} \\
\bar{G}_k +a^{s+1} G^{\zeta}_{1:k-1}
\end{pmatrix}
=
\begin{pmatrix}
G_{1:k-1} \\
\bar{\omega}^{s+2}
\end{pmatrix},
$$
and with other rows unchanged.
To conclude we simply apply Lemma \ref{lem derivative}.

\end{itemize}

\end{proof}

\section{The derivatives of the kernel at the origin}\label{sect derivatives GM}

In Section \ref{sect method} we have introduced  a method to determine the minimal null control time of systems of the form $(Q,G_{+-})$.
We now explain how this method can also be used to determine the minimal null control time of the initial system \eqref{syst} (see Problem \ref{prob 2}).
The key observation was presented in Remark \ref{rem all dep on G at x=0}: to apply or iterate the method on systems of the form $(Q,G_{+-})$, we only need to know the values of $G_{+-}$ and its derivatives at $x=0$.
We show here that these derivatives can be explicitly computed from the knowledge of values of $M$ and its derivatives at $x=0$.
First of all, it is clear that (recall \eqref{def G})
\begin{equation}\label{deriv G}
G_j^{(N)}(0)=
\frac{\partial^N K_{m+j}}{\partial x^N}(0,0) B,
\end{equation}
for every $j$ and $N \leq r+1$.
It remains to compute the derivatives of the kernel $K$ at the origin.

Below, the $(\alpha,\beta)$ entry of $K$ is denoted by $k_{\alpha\beta}$ and the $\alpha$-th row of $K$ is denoted by $K_\alpha$.
Finally, for $N \in \N$ and a smooth enough function (or matrix-valued function) $h$ defined on $\clos{\Tau}$, we denote by
$$
\def\arraystretch{1.5}
D^N h=
\begin{pmatrix}
\frac{\partial^N h}{\partial x^N} \\
\frac{\partial^N h}{\partial x^{N-1} \partial \xi} \\
\vdots \\
\frac{\partial^N h}{\partial \xi^N} \\
\end{pmatrix}.
$$

The main result of this section is the following.

\begin{theorem}\label{thm derivatives K}
For any solution $K \in C^{r+1}(\clos{\Tau})^{n \times n}$ to the kernel equations \eqref{kern equ}, we have
\begin{equation}\label{K00}
k_{\alpha\beta}(0,0)=
\frac{m^0_{\alpha \beta}(0)}{\lambda_{\alpha}(0)-\lambda_{\beta}(0)} \quad (\beta \neq \alpha),
\end{equation}
and
$$
J^{(N)}_{\alpha\beta}
(D^N k_{\alpha\beta})(0,0)
=\Phi_{\alpha\beta}^{(N)}(K_\alpha(0,0),\ldots,(D^{N-1}K_\alpha)(0,0)),
$$
for every $1 \leq \alpha,\beta \leq n$ and $1 \leq N \leq r+1$, where $J^{(N)}_{\alpha\beta} \in \R^{(N+1) \times (N+1)}$ and $\Phi_{\alpha\beta}^{(N)}:\R^{1 \times n} \times \cdots \times \R^{N \times n} \to \R^{N+1}$ are explicit (see below) and $J^{(N)}_{\alpha\beta}$ is invertible.
\end{theorem}

Let us prove the result.
We fix a row index $\alpha \in \ens{1,\ldots,n}$.
Let us denote by
$$
f_{\alpha\alpha}(x)=k_{\alpha\alpha}(x,0).
$$
Then, the kernel equations can be written as
\begin{equation}\label{kern equ bis}
\begin{dcases}
\lambda_{\alpha}(x)\px{k_{\alpha\beta}}(x,\xi)
+\pxi{k_{\alpha\beta}}(x,\xi)\lambda_{\beta}(\xi)
=-K_\alpha(x,\xi) A_{\beta}(\xi),
\\
k_{\alpha\beta}(x,x)=f_{\alpha\beta}(x) \quad (\beta \neq \alpha), \quad k_{\alpha\alpha}(x,0)=f_{\alpha\alpha}(x),
\end{dcases}
\end{equation}
for any $1 \leq \beta \leq n$, where $A_\beta$ denotes the $\beta$-th column of
$$
A=\Lambda'+M^0,
$$
and
$$
f_{\alpha\beta}(x)=
\frac{m^0_{\alpha \beta}(x)}{\lambda_{\alpha}(x)-\lambda_{\beta}(x)} \quad (\beta \neq \alpha).
$$

Let $1 \leq N \leq r+1$ and $1 \leq l \leq N$.
Taking the derivative of the first equation in \eqref{kern equ bis}, $N-l$ times with respect to $x$ and $l-1$ times with respect to $\xi$, we obtain
\begin{multline*}
\lambda_\alpha(x)\frac{\partial^N k_{\alpha\beta}}{\partial x^{N-l+1} \partial \xi^{l-1}}(x,\xi)
+\frac{\partial^N k_{\alpha\beta}}{\partial x^{N-l} \partial \xi^l}(x,\xi) \lambda_\beta(\xi)
\\
\begin{aligned}
&=
-\sum_{\sigma=0}^{l-1} \binom{l-1}{\sigma} \frac{\partial^{N-1-\sigma} K_\alpha}{\partial x^{N-l} \partial \xi^{l-1 -\sigma}}(x,\xi) A_{\beta}^{(\sigma)}(\xi)
\\
& \phantom{=} -\sum_{\sigma=1}^{N-l} \binom{N-l}{\sigma} \lambda_\alpha^{(\sigma)}(x) \frac{\partial^ {N-\sigma} k_{\alpha\beta}}{\partial x^{N-l-\sigma+1} \partial \xi^{l-1}}(x,\xi)
\\
& \phantom{=} -\sum_{\sigma=1}^{l-1} \binom{l-1}{\sigma} \frac{\partial^ {N-\sigma} k_{\alpha\beta}}{\partial x^{N-l} \partial \xi^{l-\sigma}}(x,\xi) \lambda_\beta^{(\sigma)}(\xi).
\end{aligned}
\end{multline*}
On the other hand, taking $N$ times the derivative of the equations of the second line of \eqref{kern equ bis}, we obtain
$$
\sum_{l=0}^{N} \binom{N}{l} \frac{\partial^N k_{\alpha\beta}}{\partial x^{N-l} \partial \xi^l}(x,x)
=f_{\alpha\beta}^{(N)}(x) \quad (\beta \neq \alpha),
\quad \frac{\partial^N k_{\alpha\alpha}}{\partial x^N}(x,0)=f_{\alpha\alpha}^{(N)}(x).
$$
At the origin $(x,\xi)=(0,0)$, the previous computations give the system
$$
J^{(N)}_{\alpha\beta}
\def\arraystretch{1.5}
\begin{pmatrix}
\frac{\partial^N k_{\alpha\beta}}{\partial x^N}(0,0) \\
\frac{\partial^N k_{\alpha\beta}}{\partial x^{N-1} \partial \xi}(0,0) \\
\vdots \\
\frac{\partial^N k_{\alpha\beta}}{\partial \xi^N}(0,0) \\
\end{pmatrix}
=\Phi_{\alpha\beta}^{(N)},
$$
where $J^{(N)}_{\alpha\beta}$ is the $(N+1) \times (N+1)$ matrix defined by
\begin{equation}\label{def J matrix}
J^{(N)}_{\alpha\beta}=
\begin{psmallmatrix}
\lambda_\alpha(0) & \lambda_\beta(0) & 0 & \cdots & 0\\
0 & \lambda_\alpha(0) & \lambda_\beta(0) & \ddots & \vdots \\
\vdots & \ddots & \ddots & \ddots & 0 \\
0 & \cdots & 0 & \lambda_\alpha(0) & \lambda_\beta(0) \\
\binom{N}{0} & \binom{N}{1} & \cdots & \binom{N}{N-1} & \binom{N}{N}
\end{psmallmatrix}
\quad (\beta \neq \alpha),
\quad
J^{(N)}_{\alpha\alpha}=
\begin{psmallmatrix}
\lambda_\alpha(0) & \lambda_\alpha(0) & 0 & \cdots & 0\\
0 & \lambda_\alpha(0) & \lambda_\alpha(0) & \ddots & \vdots \\
\vdots & \ddots & \ddots & \ddots & 0 \\
0 & \cdots & 0 & \lambda_\alpha(0) & \lambda_\alpha(0) \\
1 & 0 & \cdots & \cdots & 0
\end{psmallmatrix},
\end{equation}
and $\Phi_{\alpha\beta}^{(N)}$ is the $\beta$-th column of the $(N+1) \times n$ matrix
$$
\Phi_\alpha^{(N)}=
\begin{pmatrix}
\bar{\Phi}_\alpha^{(N)}
\\
F_\alpha^{(N)}(0)
\end{pmatrix},
$$
where
$$
F_\alpha=\begin{pmatrix} f_{\alpha 1} & f_{\alpha 2} & \cdots & f_{\alpha n} \end{pmatrix},
$$
and
\begin{equation}\label{def Hbar}
\begin{aligned}
\bar{\Phi}_\alpha^{(N)} &=
-\sum_{\sigma=0}^{N-1} R^{(N)}_\sigma (D^{N-1-\sigma}K_\alpha)(0,0) A^{(\sigma)}(0)
\\
& \phantom{=} -\sum_{\sigma=1}^{N-1} \lambda_\alpha^{(\sigma)}(0) \check{S}^{(N)}_\sigma (D^{N-\sigma} K_{\alpha})(0,0)
\\
& \phantom{=} -\sum_{\sigma=1}^{N-1} S^{(N)}_\sigma (D^{N-\sigma} K_{\alpha})(0,0) \Lambda^{(\sigma)}(0),
\end{aligned}
\end{equation}
in which $R^{(N)}_0=E^{(N)}_0$ and
$$
R^{(N)}_\sigma=
\begin{pmatrix}
0 \\
E^{(N)}_\sigma
\end{pmatrix},
\quad
\check{S}^{(N)}_\sigma=
\begin{pmatrix}
\check{E}^{(N)}_\sigma & 0 \\
0 & 0
\end{pmatrix},
\quad
S^{(N)}_\sigma=
\begin{pmatrix}
0 & 0 \\
0 & E^{(N)}_\sigma
\end{pmatrix},
\quad (1 \leq \sigma \leq N-1),
$$
with
$$
E^{(N)}_\sigma =
\diag\left(\binom{i}{\sigma}\right)_{\sigma \leq i \leq N-1}
\quad
\check{E}^{(N)}_\sigma =
\diag\left(\binom{N-i}{\sigma}\right)_{1 \leq i \leq N-\sigma}.
$$

%
%
%
%
%
%

Finally, the matrix $J^{(N)}_{\alpha\beta}$ is invertible and its inverse is easy to compute:

\begin{proposition}\label{comp inv J}
For any $\beta$, the matrix $J^{(N)}_{\alpha\beta}$ is invertible and, denoting by $(J^{(N)}_{\alpha\beta})^{-1}_l$ the $l$-th row of the inverse of $J^{(N)}_{\alpha\beta}$, we have
$$
(J^{(N)}_{\alpha\alpha})^{-1}_1=e_{N+1}^\tr,
\quad
(J^{(N)}_{\alpha\beta})^{-1}_1=
\begin{pmatrix}
a_1 & \cdots & a_{N+1}
\end{pmatrix} \quad (\beta \neq \alpha),
$$
with
$$
a_i = -\frac{1}{\lambda_\beta(0)} a_{N+1} \sum_{j=0}^{N-i} \binom{N}{i+j} \left(-\frac{\lambda_\alpha(0)}{\lambda_\beta(0)}\right)^j, \quad 1 \leq i \leq N,
\quad a_{N+1} =\left(1-\frac{\lambda_\alpha(0)}{\lambda_\beta(0)}\right)^{-N},
$$
and then (for any $\beta$),
$$
(J^{(N)}_{\alpha\beta})^{-1}_l=\frac{1}{\lambda_\beta(0)} e_{l-1}^\tr
-\frac{\lambda_\alpha(0)}{\lambda_\beta(0)} (J^{(N)}_{\alpha\beta})^{-1}_{l-1}, \quad 2 \leq l \leq N+1.
$$
\end{proposition}

This concludes the proof of Theorem \ref{thm derivatives K}.

\begin{remark}\label{rem uniqueness of values}
As already mentioned above, the kernel equations \eqref{kern equ} have many $C^{r+1}$ solutions.
The same is true if we require in addition that
$$k_{\alpha\alpha}(x,0)=f_{\alpha\alpha}(x),$$
for every $1 \leq \alpha \leq n$ and $x \in [0,1]$, where $f_{\alpha\alpha}$ are now given functions in $C^{r+1}([0,1])$.
However, we see from Theorem \ref{thm derivatives K} that, despite this lack of uniqueness, once the values of the function $f_{\alpha\alpha}$ and its derivatives are known at $x=0$, then the values of $K_\alpha$ and its derivatives at $(x,\xi)=(0,0)$ are uniquely determined.
\end{remark}

\begin{remark}
If $\Lambda$ is constant, then the last two sums in $\bar{\Phi}_\alpha^{(N)}$ disappear (see \eqref{def Hbar}).
If, in addition, $M^0$ is also constant, then so are $A,F$ and we obtain the simple formula
$$
\Phi_\alpha^{(N)}=
\begin{pmatrix}
-E^{(N)}_0 (D^{N-1}K_\alpha) M^0
\\
0
\end{pmatrix}.
$$
\end{remark}

\section{Summary and applications}\label{sect app}

Our method to determine the minimal null control time of system \eqref{syst} is thus as follows.

First of all, we fix a solution $K$ to the kernel equations \eqref{kern equ}.
This solution can be arbitrary or we may require additional constraints if this helps in the computations (see Remark \ref{rem uniqueness of values}).
After that, $G$ is uniquely defined by \eqref{def G}.

Then, we do the following, successively for each $k=1,2,\ldots,\nmin$ (with $Q^{[0]}=Q$ and $G^{[0]}_{+-}=G_{+-}$):
\begin{itemize}
\item
We check if $Q^{[k-1]}_k$ and the first $k-1$ rows of $Q^{[k-1]}$ are linearly independent (resp. $Q_1 \neq 0$ for $k=1$).
If that is the case, we just set $Q^{[k]}=Q^{[k-1]}$ and $G_{+-}^{[k]}=G_{+-}^{[k-1]}$.

\item
Otherwise, we find the first integer $s$ ($\leq r$) such that the assumption of Theorem \ref{main thm G} is satisfied and such that $\omega^{s+1}$ is linearly independent with the first $k-1$ rows of $Q^{[k-1]}$ (resp. $\gamma^{s+1} \neq 0$ for $k=1$).
That such an integer exists is the main assumption of our method.
In particular, we need enough regularity on $G_{+-}$ (and thus on $\Lambda, M$).
Now, to compute $\omega^l$ for $1 \leq l \leq s+1$, we proceed as follows:

\begin{itemize}
\item
We first express $\omega^l$ in function of $G_1,\ldots,G_k$ and their derivatives at $x=0$, using its definition \eqref{def omega}-\eqref{def omegabar}, the definition of the operator $\opD_k$ (see Lemma \ref{lem derivative}) and the definition of $G^{\zeta}_{1:k-1}$ (see \eqref{def Gzeta}).

\item
To compute $G_1,\ldots,G_k$ and their derivatives at $x=0$, we use identity \eqref{deriv G} and the formula provided by Theorem \ref{thm derivatives K} (with Proposition \ref{comp inv J}).
\end{itemize}
We then define $Q^{[k]}$ and $G^{[k]}_{+-}$ as in the conclusion of Theorem \ref{main thm G}.

\end{itemize}

After repeating the previous procedure successively for each $k$ we obtain in the end a new system $(Q^{[\nmin]}, G^{[\nmin]}_{+-})$ which has the same minimal null control time as the initial system $(Q,G_{+-})$, and such that the first $\nmin$ rows of $Q^{[\nmin]}$ are linearly independent.
Then, we conclude with Theorem \ref{thm comp time} that
$$\Tinf=\max_{1 \leq k \leq \nmin} \ens{T_{m+k}+T_{c_k}, \quad T_m},$$
where the indices $c_k$ refer to $Q^{[\nmin]}$ (and are easy to compute).

Let us detail an example to help clarify this procedure.
We will then present two other applications.

\subsection{A detailed example}

Let us consider the following $3+2$ system with constant coefficients:
\begin{equation}\label{syst ex}
\Lambda=\diag\left(-2,-1,-\frac{1}{2},1,2\right)
\quad
M=\left(\begin{array}{ccc|cc}
0 & 0 & 0 & 1 & 6 \\
0 & 0 & 0 & 2 & 1 \\
0 & 0 & 0 & 3 & -1 \\
\hline
3 & 2 & 1 & 0 & 0 \\
8 & 9 & -\frac{20}{3} & 0 & 0
\end{array}\right),
\quad
Q=
\begin{pmatrix}
0 & 1 & -2 \\
0 & 2 & -4
\end{pmatrix}.
\end{equation}
We have
$$
T_1=\frac{1}{2}, \quad T_2=1, \quad T_3=2, \quad T_4=1, \quad T_5=\frac{1}{2}.
$$

The current results in the literature are insufficient to determine the minimal null control time of this system.
Indeed, the result of \cite{Rus78} gives that system \eqref{syst ex} is null controllable in time
$$T=T_4+T_3=3,$$
and the result of \cite{HO22-JDE} (see \eqref{max Tinf}) gives that system \eqref{syst ex} is null controllable in any time
$$T>\max\ens{T_4+T_2, T_5+T_3}=\frac{5}{2}.$$
Also, note that $Q \not\in \setCN$ (recall \eqref{def setCN}), so that the results of \cite{CN19,CN21} cannot be used.

With our method, we will prove that system \eqref{syst ex} is in fact null controllable in any time $T>2$.
We will also see that this system is not null controllable if $T<2$.

\begin{proposition}
The minimal null control time of system \eqref{syst ex} is $\Tinf=2$.
\end{proposition}

\begin{proof}
We apply our method.
Below, $K$ is any $C^2$ solution in $\clos{\Tau}$ to the corresponding kernel equations \eqref{kern equ} that satisfies in addition the conditions
$$k_{44}(x,0)=k_{55}(x,0)=0, \quad \forall x \in [0,1].$$

\begin{enumerate}[label={[\arabic*]}]
\item
The first row of the boundary coupling matrix is
$$
Q_1=
\begin{pmatrix}
0 & 1 & -2
\end{pmatrix}.
$$
It is nonzero, so that there is nothing to do and we move on to the second row.

\item
The second row of the boundary coupling matrix is
$$
Q_2=
\begin{pmatrix}
0 & 2 & -4
\end{pmatrix}.
$$
Clearly, $Q_2$ and $Q_1$ are linearly dependent:
$$Q_2 + a^0 Q_1=0, \quad a^0=-2.$$

\begin{enumerate}[(1)]
\item
We thus form the row vector $\omega^1$ (see \eqref{def omega}-\eqref{def omegabar}):
\begin{align*}
\omega^1
&=
G_2(0) +a^0 G^{\zeta}_{1:1}(0)
\\
&=G_2(0)- 2G^{\zeta}_{1:1}(0),
\end{align*}
where $G^{\zeta}_{1:1}(x)=G_1(x/2)$ (see \eqref{def Gzeta}).
We need to check if this new row is linearly independent with $Q_1$.
We first express it in function of $G_2(0)$ and $G_1(0)$ only:
$$\omega^1=G_2(0)-2G_1(0).$$
We now have to compute $G_2(0)$ and $G_1(0)$.
From the definition \eqref{def G} of $G$, we have
$$
G_2(0)=K_5(0,0)B,
\quad
B=
\left(\begin{array}{ccc}
2 & 0 & 0 \\
0 & 1 & 0 \\
0 & 0 & \frac{1}{2} \\
\hline
0 & -1 & 2 \\
0 & -4 & 8
\end{array}\right).
$$
Using that $k_{55}(\cdot,0)=0$, we clearly have (see e.g. \eqref{K00})
$$
K_5(0,0)=\left(\begin{array}{ccc|cc}
2 & 3 & -\frac{8}{3} & 0 & 0
\end{array}\right).
$$
Thus,
$$
G_2(0)=
\begin{pmatrix}
4 & 3 & -\frac{4}{3}
\end{pmatrix}.
$$
Similarly, using that $k_{44}(\cdot,0)=0$, we get
$$
K_4(0,0)=\left(\begin{array}{ccc|cc}
1 & 1 & \frac{2}{3} & 0 & 0
\end{array}\right),
$$
and thus
$$
G_1(0)=
\begin{pmatrix}
2 & 1 & \frac{1}{3}
\end{pmatrix}.
$$
As a result,
$$
\omega^1=
\begin{pmatrix}
0 & 1 & -2
\end{pmatrix}.
$$
This new row is again linearly dependent with $Q_1$:
$$\omega^1 + a^1 Q_1=0, \quad a^1=-1.$$
We thus have to go on.

\item
We form the row vector $\omega^2$ (see \eqref{def omega}-\eqref{def omegabar}):
\begin{align*}
\omega^2
&=(\opD_2 G_2)(0) +a^0 (\opD_2 G^{\zeta}_{1:1})(0) +a^1 G^{\zeta}_{1:1}(0)
\\
&=(\opD_2 G_2)(0) -2 (\opD_2 G^{\zeta}_{1:1})(0) -
G^{\zeta}_{1:1}(0),
\end{align*}
where $\opD_2=2 \ddx$.
We need to check if this new row is linearly independent with $Q_1$.
We first express it in function of $G_2'(0)$ and $G_1'(0), G_1(0)$ only:
$$
\omega^2
=2 G_2'(0) -2 G_1'(0) -G_1(0).
$$
We now have to compute $G_2'(0)$ and $G_1'(0)$ (recall that $G_1(0)$ was computed in the previous step).
From the definition \eqref{def G} of $G$, we have
$$
G_2'(0)=\px{K_5}(0,0)B.
$$
Using the formula of Theorem \ref{thm derivatives K} (with Proposition \ref{comp inv J}), we obtain that
$$
\px{K_5}(0,0)=0.
$$
Therefore,
$$
G_2'(0)=0.
$$
Similarly, we get
$$
\px{K_4}(0,0)=
\theta e_5^\tr,
$$
with $\theta=19/3$, and thus
$$
G_1'(0)=-4\theta Q_1.
$$
As a result,
$$
\omega^2=
8\theta Q_1 - G_1(0).
$$
Clearly, $G_1(0)$ and $Q_1$ are linearly independent and thus so are $\omega^2$ and $Q_1$.
We thus stop here.

\end{enumerate}

\end{enumerate}

In conclusion, the new boundary coupling matrix is
$$
Q^{[2]}=
\begin{pmatrix}
Q_1 \\
\omega^2
\end{pmatrix}.
$$
It is full row rank.
The minimal null control time of system \eqref{syst ex} is thus
$$\Tinf=\max \ens{T_4+T_{c_1}, \quad T_5+T_{c_2}, \quad T_3},$$
where $c_1,c_2$ refer to $Q^{[2]}$.
We easily check that the canonical form (Definition \ref{def can form}) of $Q^{[2]}$ is
$$
\begin{pmatrix}
0 & 1 & 0 \\
1 & 0 & 0
\end{pmatrix},
$$
so that $c_1=2$, $c_2=1$, and
$$\Tinf=\max\ens{T_4+T_2, \quad T_5+T_1, \quad T_3}=2.$$

\end{proof}

\subsection{No boundary coupling}

Consider system \eqref{syst} with $Q=0$.
At the moment, the only result available in the literature is the one of \cite{Rus78}, and it gives that the system is null controllable in time $T_{m+1}+T_m$.

A simple application of our method gives the following result:

\begin{proposition}
Assume that $Q=0$, $\Lambda, M \in C^1$ and
$$\rank Q'=\nmin, \quad \text{ where } Q'=\left(\frac{m_{\alpha \beta}(0)}{\lambda_{\alpha}(0)-\lambda_{\beta}(0)}\right)_{\substack{\alpha \geq m+1 \\ \beta \leq m}},$$
(we recall that $\nmin=\min\ens{p,m}$).
Then, the minimal null control time is
$$
\Tinf=\max_{1 \leq k\leq \nmin} \ens{T_{m+k}+T_{c_k}, \quad T_m},
$$
where the indices $c_k$ refer to $Q'$.
\end{proposition}

\begin{proof}
We apply our method.
Below, $K$ is any $C^1$ solution in $\clos{\Tau}$ to the corresponding kernel equations \eqref{kern equ}.

\begin{enumerate}[label={[\arabic*]}]
\item
The first row of the boundary coupling matrix is zero.
We thus form the row vector $\gamma^1=G_1(0)$ and check whether it is still zero or not.
From the definition \eqref{def G} of $G$, we have
$$G_1(0)=-K_{m+1,-}(0,0)\Lambda_-(0),$$
where $K_{m+1,-}=\begin{pmatrix} k_{m+1,1} & \cdots & k_{m+1,m} \end{pmatrix}$ and $\Lambda_-=\diag(\lambda_1,\ldots,\lambda_m)$.
Clearly, $K_{m+1,-}(0,0) =Q'_1$ (see e.g. \eqref{K00}).
Thus,
$$G_1(0)=-Q'_1 \Lambda_-(0).$$
This row vector is not zero by assumption on $Q'$.
We set
$$
Q^{[1]}=
\begin{pmatrix}
-Q'_1 \Lambda_-(0) \\
0 \\
\vdots \\
0
\end{pmatrix},
\quad
G^{[1]}_{+-}=
\begin{pmatrix}
\opD_1G_1 \\
G_2 \\
\vdots \\
G_p
\end{pmatrix},
$$
and we move on to the next row.

\item
The second row of the new boundary coupling matrix is zero.
It is thus trivially linearly dependent with its first row:
$$Q^{[1]}_2+a^0 Q^{[1]}_1=0, \quad a^0=0.$$
We thus form the row vector $\omega^1$ (see \eqref{def omega}-\eqref{def omegabar}):
\begin{align*}
\omega^1
&=
G^{[1]}_2(0) +a^0 G^{[1]\zeta}_{1:1}(0)
\\
&=G^{[1]}_2(0)
\\
&=G_2(0).
\end{align*}
As before, we easily check that
$$G_2(0)=-Q'_2 \Lambda_-(0).$$
Using the assumption on $Q'$, it is linearly independent with $Q^{[1]}_1$.

\end{enumerate}

Repeating this procedure for the next rows, we easily see that the new boundary coupling matrix that we obtain in the end is
$$Q^{[\nmin]}=-Q' \Lambda_-(0).$$
By assumption, it is full rank.
The minimal null control time of system \eqref{syst ex} is thus
$$
\Tinf=\max_{1 \leq k\leq \nmin} \ens{T_{m+k}+T_{c_k}, \quad T_m},
$$
where the indices $c_k$ refer to $Q^{[\nmin]}$.
Finally, it is clear that $Q^{[\nmin]}$ and $Q'$ have the same canonical form (Definition \ref{def can form}), so that the indices $c_k$ are also those of $Q'$.

\end{proof}

\subsection{Result for $m+1$ systems}\label{sect m1 syst}

Consider system \eqref{syst} with $p=1$.
In that case, the boundary coupling matrix $Q$ is just a row vector.
From \cite{HO21-JMPA}, we know that, if $Q \neq 0$, then
$$\Tinf=\max\ens{T_{m+1}+T_c, \quad T_m},$$
where $c$ is the index of the first nonzero entry of $Q$ (in \cite{CN19,CN21}, the authors showed that $\Tinf \leq \max\ens{T_{m+1}+T_1, \quad T_m}$ if the first entry of $Q$ is not zero).
On the other hand, if $Q=0$, then, as mentioned in the previous section, the only result available is the one of \cite{Rus78} and it gives that the system is null controllable in time $T_{m+1}+T_m$.

We will complete these results by showing the following:

\begin{proposition}\label{prop p=1}
Assume that $p=1$, $Q=0$ and $\Lambda, M \in C^{r+1}$ with
$$
M_{m+1,-}^{(r)}(0) \neq 0,
$$
for some integer $r \in \N$, where $M_{m+1,-}=\begin{pmatrix} m_{m+1,1} & \cdots & m_{m+1,m} \end{pmatrix}$, and let $r_0$ be the smallest such integer.
Then, the minimal null control time is
$$
\Tinf=\max\ens{T_{m+1}+T_c, \quad T_m},
$$
where $c$ is the index of the first nonzero entry of $M_{m+1,-}^{(r_0)}(0)$.

\end{proposition}

\begin{proof}
We apply our method.
Below, $K$ is any $C^{r+1}$ solution in $\clos{\Tau}$ to the corresponding kernel equations \eqref{kern equ} that satisfies in addition the condition
$$k_{m+1,m+1}(x,0)=0, \quad \forall x \in [0,1].$$

We form the row vectors $\gamma^l=G_1^{(l-1)}(0)$, $l \geq 1$, and check whether they are zero or not.
To this end, we will first prove the following property: for any $N \in \ens{0,\ldots,r}$, if
\begin{equation}\label{hyp G1 zero}
G_1^{(l)}(0)=0, \quad \forall l \leq N-1,
\end{equation}
(no assumption if $N=0$) then, for any $\beta \leq m$,
$$
g_{m+1,\beta}^{(N)}(0)=0
\quad \Longleftrightarrow \quad
m_{m+1,\beta}^{(N)}(0)=0.
$$
We prove it by induction on $N$.
We start with $N=0$.
From the definition \eqref{def G} of $G$, we have
$$G_1(0)=K_{m+1}(0,0)B,
\quad
B=
-\Lambda(0)
\begin{pmatrix}
\Id_m \\
0
\end{pmatrix}.
$$
Clearly, for any $\beta \leq m$,
$$
g_{m+1,\beta}(0)=0
\quad \Longleftrightarrow \quad
k_{m+1,\beta}(0,0)=0
\quad \Longleftrightarrow \quad
m^0_{m+1,\beta}(0)=0
\quad \Longleftrightarrow \quad
m_{m+1,\beta}(0)=0,
$$
(see \eqref{K00} and recall \eqref{def M0}).
If $r \geq 1$, fix now $N \in \ens{1,\ldots,r}$, assume that the property holds for all integers before $N-1$ and let us prove it for $N$.
From the definition \eqref{def G} of $G$, we have
$$G_1^{(N)}(0)=\frac{\partial^N K_{m+1}}{\partial x^N}(0,0)B.$$
By assumption, we have \eqref{hyp G1 zero}, and it follows from the induction assumption that $(M_{m+1}^0)^{(l)}(0)=0$ for every $l \leq N-1$ (recall that $m_{m+1,m+1}^0=0$ by definition of $M^0$).
Using that $k_{m+1,m+1}(\cdot,0)=0$, we deduce from Theorem \ref{thm derivatives K} that
$$K_{m+1}(0,0)=\ldots=(D^{N-1} K_{m+1})(0,0)=0.$$
Still by the same theorem, we also obtain that
$$
J^{(N)}_{m+1,\beta}
(D^N k_{m+1,\beta})(0,0)
=
\begin{pmatrix}
0 \\
\vdots \\
0 \\
f_{m+1,\beta}^{(N)}(0)
\end{pmatrix}.
$$
It then follows from the structure of $J^{(N)}_{m+1,\beta}$ (see \eqref{def J matrix}) that
$$
\frac{\partial^N k_{m+1,\beta}}{\partial x^N}(0,0)=0
\quad \Longleftrightarrow \quad
(D^N k_{m+1,\beta})(0,0)=0
\quad \Longleftrightarrow \quad
f_{m+1,\beta}^{(N)}(0)=0.
$$
Finally, since $M_{m+1,-}^{(s)}(0)=0$ for every $s \leq N-1$, it is clear that, for $\beta \leq m$,
$$
f_{m+1,\beta}^{(N)}(0)=0
\quad \Longleftrightarrow \quad
(m^0_{m+1,\beta})^{(N)}(0)=0
\quad \Longleftrightarrow \quad
m_{m+1,\beta}^{(N)}(0)=0.
$$
This proves the induction.

By assumption of the proposition and the property we have just established, we have $\gamma^l=0$ for every $l \in \ens{0,\ldots,r_0}$ and $\gamma^{r_0+1} \neq 0$.
As a result, the new boundary coupling matrix is
$$
Q^{[1]}=\lambda_{m+1}(0)^{r_0} G_1^{(r_0)}(0).
$$
It is a nonzero row vector.
Therefore, the minimal null control time of the system is
$$\Tinf=\max\ens{T_{m+1}+T_{c_1}, \quad T_m},$$
where $c_1$ refer to $Q^{[1]}$.
Clearly, the canonical form (Definition \ref{def can form}) of $Q^{[1]}$ is $e_c^\tr$, where $c$ is the index of the first nonzero entry of $G_1^{(r_0)}(0)$ (equivalently, of $M_{m+1,-}^{(r_0)}(0)$), so that $c_1=c$.

\end{proof}

\begin{remark}
If $\Lambda, M \in C^{\infty}$ and $M_{m+1,-}$ is analytic in $(-\epsilon,1)$ for some $\epsilon>0$, then either the assumption of Proposition \ref{prop p=1} holds or we have $M_{m+1,-}=0$.
In that second situation it is not difficult to directly check that
$$\Tinf=\max\ens{T_m, T_{m+1}}.$$
Indeed, in that case the last component of the system is not coupled with the other ones and it can thus be considered as a source term for the reduced system formed from the first $m$ equations.
Since this reduced system is exactly controllable in time $T_m$ without source term, the same remains true with a source term.
The problem of finding $\Tinf$ is thus completely solved in the analytic framework for $p=1$.
\end{remark}

\section*{Acknowledgments}

The first author would like to thank the Institute of Mathematics of the Jagiellonian University for its hospitality.
This work was initiated while he was visiting there.
This project was supported by National Natural Science Foundation of China (Nos. 12471421 and 12122110) and National Science Centre, Poland UMO-2023/50/E/ST1/00081 and UMO-2020/39/D/ST1/01136.
For the purpose of Open Access, the authors have applied a CC-BY public copyright licence to any Author Accepted Manuscript (AAM) version arising from this submission.


\bibliographystyle{amsalpha}
\bibliography{biblio}

\end{document}